\documentclass[12pt]{article}%
\usepackage[intlimits]{amsmath}
\usepackage{amssymb}
\usepackage{latexsym}
\usepackage{tikz}
\usepackage{anysize}
\usepackage[T1]{fontenc}
\usepackage[sc]{mathpazo}
\usepackage{color}
\usepackage[colorlinks]{hyperref}
\usepackage{amsfonts}
\usepackage{graphicx}
\usepackage{float}
\usepackage{subfig}
\usepackage{comment}
\definecolor {refcol}{RGB}{40,0,255}
\hypersetup{colorlinks=true,allcolors=refcol}
\makeatletter
\newfont{\footsc}{cmcsc10 at 8truept}
\newfont{\footbf}{cmbx10 at 8truept}
\newfont{\footrm}{cmr10 at 10truept}
\newtheorem{theorem}{Theorem}

\newtheorem{corollary}{Corollary}
\newtheorem{Res}{Result}

\newtheorem{definition}{Definition}
\newtheorem{example}{Example}

\newtheorem{lemma}{Lemma}

\newtheorem{proposition}[theorem]{Proposition}

\newtheorem{remark}{Remark}

\newenvironment{proof}[1][Proof]{\noindent{\textbf {#1}  }}  {\hfill$\Box$\bigskip}

\begin{document}

\title{Spectral properties of the deformed Laplacian matrix of trees and \emph{H-join} graphs}

\author{Roberto C. D\'{i}az$^1$, ~~\ Elismar R. Oliveira$^2$, ~~\ Vilmar Trevisan$^2$\\
\small{$^{1}${\it Departamento de Matematicas, Universidad Cat\'olica del Norte,}} \\
\small{\noindent Avenida Angamos 0610, Antofagasta, Chile} \\
\small{\noindent  \texttt{rdiaz01@ucn.cl}} \\
\small{$^{2}${\it Instituto de Matem\'atica e Estat\'{\i}stica, Universidade Federal do Rio Grande do Sul,}} \\
\small{Porto Alegre, Brasil} \\
\small{\texttt{elismar.oliveira@ufrgs.br,~~trevisan@mat.ufrgs.br}}
}
\date{}
\maketitle
\begin{abstract}
This paper investigates  spectral properties of the deformed Laplacian matrix, which merges the Laplacian and signless Laplacian matrices of a graph through a one-parameter family of matrices. We present general results on the eigenvalues of these matrices for simple undirected graphs. Additionally, we analyze the spectrum of the deformed Laplacian in the specific cases of trees and H-join graphs. For trees, we derive strong results on the localization of eigenvalues, while for H-join graphs, we explicitly compute the spectrum of the deformed Laplacian.

\end{abstract}


\section{Introduction}
The concept of the \emph{deformed Laplacian} matrix was introduced by F. Morbidi in 2013~\cite{Mor2013}, within the context of a generalized continuous-time consensus protocol. In this framework, the traditional Laplacian matrix of the communication graph is replaced by a deformed version. The stability of this modified protocol is then analyzed by exploring the role of a real parameter $s$, using spectral theory techniques from quadratic eigenvalue problems, applied to both specific families of graphs (undirected and directed) and more general structures. Theoretical insights are supported by examples and simulations.

In a different but related direction, the authors of~\cite{GDV2018} proposed a novel centrality measure based on nonbacktracking walks — walks that avoid revisiting nodes in patterns such as $uvu$, where $u$ and $v$ are distinct adjacent vertices. This approach is argued to yield more informative rankings than traditional walk-based centralities. Interestingly, the authors demonstrate that this Katz-style centrality can be computed using a form of the deformed Laplacian matrix. Under certain conditions, this measure aligns with the nonbacktracking variant of eigenvector centrality introduced in a 2014 study.

Additionally, a distinct formulation of the deformed Laplacian appeared in~\cite{DCT2019}, where the spectral properties of another variation were examined in the context of community detection in sparse networks. This highlights the versatility and growing relevance of deformed Laplacians across different applications in network science.

Let $G$ be a simple undirected graph with $n$ vertices. Denote by $A$ its adjacency matrix, by $D$ the diagonal degree matrix, and by $I$ the $n \times n$ identity matrix. For a real scalar $s$, the deformed Laplacian matrix of $G$ is defined as:
\begin{align}\label{MG}
M_G(s) &= I - sA + s^2(D - I) \notag \\ \notag \\
&= \left[
    \begin{array}{ccccc}
      1+ s^2 (d_{v_1}-1) & -s a_{1 \, 2} & \cdots & -s a_{1\, n-1} & -s a_{1 \, n} \\
      -s a_{1 \, 2} & 1 + s^2 (d_{v_2}-1) & \cdots & -s a_{2\, n-1} & -s a_{2 \,n} \\
      \vdots & \vdots & \ddots & \vdots & \vdots \\
      -s a_{n-1\,1} & -s a_{n-1\,2} & \cdots & 1+ s^2 (d_{v_{n-1}}-1) & -s a_{n-1\,n} \\
      -s a_{n \,1} & -s a_{n\,2} & \cdots & -s a_{n\,n-1} & 1+ s^2 (d_{v_n}-1) \\
    \end{array}
  \right].
\end{align}

The main objective of this paper is to explore the spectral properties of the deformed Laplacian matrix within a general graph-theoretical framework. While this matrix originally emerged in the context of consensus protocols in network dynamics, our aim is to study its mathematical structure from a broader spectral perspective. We focus, in particular, on the behavior of its eigenvalues and spectral radius, deriving specific and detailed results for classes of graphs such as trees. To achieve this, we employ eigenvalue localization techniques from~\cite{JT2011},~\cite{livro}, and~\cite{OliveTrevAppDiff}, which have proven to be powerful tools in the spectral analysis of symmetric matrices associated with graphs.

Our motivation is twofold. First, it stems from the growing interest in parametrized families of matrices that interpolate between classical graph matrices. A striking example is the work of V. Nikiforov~\cite{nikiforov2017Aalpha}, who introduced the $A_\alpha$ matrix of a graph, defined as $A_G(\alpha) := \alpha D(G) + (1 - \alpha) A(G)$, where $0 \leq \alpha \leq 1$. This formulation provides a smooth transition between the adjacency matrix $A(G)$ and the signless Laplacian $Q(G) = D(G) + A(G)$, and has opened new directions in spectral graph theory.

In a similar spirit, the deformed Laplacian matrix can be interpreted as a bridge between the standard Laplacian matrix $L = D - A$ and the signless Laplacian. This interpolation offers a promising avenue for generalizing and unifying spectral results associated with both matrices. Moreover, the structure of the deformed Laplacian reveals new algebraic features that may have applications not only in theoretical studies, but also in practical areas such as network robustness, centrality measures, and graph-based machine learning.

Therefore, our investigation aims not only to extend the theoretical understanding of this matrix, but also to contribute to the broader program of uncovering the interplay between combinatorial structure and spectral properties in graphs.

Part of this paper is also devoted to determining the spectrum of the deformed Laplacian matrix for graphs constructed via the $H$-join operation~\cite{H2}, applied to a family of regular graphs. This construction allows for the synthesis of complex graph topologies from simpler building blocks, making it a valuable tool for both theoretical investigations and practical applications. By exploiting the regularity of the constituent graphs and the structure imposed by the $H$-join, we are able to derive explicit expressions for the eigenvalues of the resulting deformed Laplacian matrix. These spectral characterizations not only enrich the theoretical understanding of the $H$-join operation in the context of deformed Laplacians, but also highlight the interplay between local regularity and global structure in determining spectral properties. The analytical results obtained in this section contribute to the broader effort of identifying graph classes for which the spectrum of the deformed Laplacian can be computed in closed form.

\bigskip

\noindent The structure of the paper is organized as follows:
In Section~\ref{sec:preliminaries}, we review key definitions and notational conventions.
Section~\ref{sec:The spectral radius of a subgraph} focuses on the spectral radius, showing that it does not increase when passing to subgraphs.
In Section~\ref{sec:Some bounds on the spectral radius of the deformed Laplacian}, we present general upper bounds for the spectral radius of the deformed Laplacian.
Section~\ref{sec:Spectrum of deformed Laplacian associated to a tree} analyzes its behavior on trees, drawing parallels with classical Laplacian and signless Laplacian matrices in bipartite graphs.
Section~\ref{sec:Spectrum of the deformed Laplacian associated to a H-join graph} concludes by determining the spectrum of the deformed Laplacian for $H$-join graphs.

\section{Preliminaries}\label{sec:preliminaries}

Let $G=\left( V\left( G\right) ,E\left( G\right) \right)$ be a simple undirected graph of order $n$. Let us denote by $\sigma(M)=\left\{\lambda^{[m_{1}]}_{1},\lambda^{[m_{2}]}_{2},\ldots,\lambda^{[m_{r}]}_{r}\right\}$, the spectrum (the multiset of eigenvalues) of a square matrix $M$, where $m_{i}$ stands for the multiplicity of eigenvalues $\lambda_{i}$, $i=1,\ldots,r$ An eigenvalue of a matrix $M$ will be denoted by $\lambda(M)$. Given two disjoint vertex graphs $G_{1}$ and $G_{2},$ the {\it join} of $G_{1}$ and $G_{2}$, denoted by $G=G_{1}\vee G_{2}$, has a vertex set $V\left( G\right)=V\left( G_{1}\right) \cup V\left( G_{2}\right)$ and $E\left( G\right) = E(G_{1})\cup E(G_{2}) \cup \left\{ xy:x\in V\left(G_{1}\right) ,y\in V\left( G_{2}\right) \right\}.$

This join operation can be generalized as follows \cite{H1,H2}: Let $H=(V(H),E(H))$ be a graph of order $r$, with vertex set $V(H)=\{1,\ldots,r\}$, and let $\left\{G_{1},\ldots ,G_{r}\right\} $ be a set of pairwise vertex disjoint graphs. For $1 \leq i \leq r$, the vertex $i\in V(H)$ is assigned to the graph $G_{i}$. Now let $G$ be the graph obtained from the graphs $G_{1},\ldots ,G_{r}$ and the edges connecting each vertex of $G_{i}$ with all the vertices of $G_{j}$ if and only if $ij\in E\left( H\right).$ That is,
$G$ is the graph with vertex set $V(G)=\bigcup_{i=1}^{r}{V(G_{i})}$ and edge set
\begin{equation*}
E(G)=\left( \bigcup_{i=1}^{r}{E(G_{i})}\right) \cup \left( \bigcup_{ij\in E\left( H\right) }{\{uv:u\in V(G_{i}),v\in V(G_{j})\}}\right).
\end{equation*}
This graph operation was introduced in \cite{H1} denominated $H-join$ of the graphs $G_{1},\ldots ,G_{r}$ and denoted by
\begin{equation*}
G=\bigvee\limits_{H}{\{G_{i}:1\leq i\leq r\}}.
\end{equation*}

It is clear that if each $G_{i}$ is a graph of order $n_{i}$, then the $H-join$ obtained $G$ is a graph of order $n=n_{1}+n_{2}+\ldots +n_{r}$ (for more details on the $H-$join operation, see \cite{Schwenk74, Stevanovic09}).

Returning to the deformed Laplacian matrix $M_G(s)$, it is easy to see that for a fixed graph $G$,  $M_G(0)=I$, $M_G(1)=I - A + (D-I)=D-A=L(G)$ is the combinatorial Laplacian matrix of $G$ and $M_G(-1)=I + A + (D-I)=D+A=Q(G)$ is the signless Laplacian matrix of $G$. As we shall see later, the intervals $(-1,0)$, $(0,1)$ and their complements play a very different role with respect to the spectral properties of the deformed Laplacian matrix. This motivates the following notation for $s$.
\begin{definition}\label{def:adapted}
  We say that $s\in \mathbb{R}$ is sub-Laplacian if $|s|<1$. Analogously, $s$ is super-Laplacian if $|s|>1$. Given $\lambda>1$ a real number, we say that $s$ is adapted to $\lambda$ if $\lambda>(1+|s|)^2$.
\end{definition}

We also notice that this is a symmetric matrix underlying the graph $G$, that is, a weighted matrix of $G$. Thus, we can use an algorithm to locate eigenvalues such as Diagonalize from \cite{JT2011} to study its eigenvalues when $G$ is a tree. The Diagonalize algorithm is given below.

{ \tt
\begin{tabbing}
aaa\=aaa\=aaa\=aaa\=aaa\=aaa\=aaa\=aaa\= \kill
     \> Input: matrix $M = (m_{ij})$ with underlying tree $T$\\
     \> Input: Bottom up ordering $v_1,\ldots,v_n$ of $V(T)$\\
     \> Input: real number $x$ \\
     \> Output: diagonal matrix $D = \mbox{diag}(d_1, \ldots, d_n)$ congruent to $M + xI$ \\
     \> \\
     \>  Algorithm $\mbox{Diagonalize}(M, x)$ \\
     \> \> initialize $d_i := m_{ii} + x$, for all $i$ \\
     \> \> {\bf for } $k = 1$ to $n$ \\
     \> \> \> {\bf if} $v_k$ is a leaf {\bf then} continue \\
     \> \> \> {\bf else if} $d_c \neq 0$ for all children $c$ of $v_k$ {\bf then} \\
     \> \> \>  \>   $d_k := d_k - \sum \frac{(m_{ck})^2}{d_c}$, summing over all children of $v_k$ \\
     \> \> \> {\bf else } \\
     \> \> \> \> select one child $v_j$ of $v_k$ for which $d_j = 0$  \\
     \> \> \> \> $d_k  := -\frac{(m_{jk})^2}{2}$ \\
     \> \> \> \> $d_j  :=  2$ \\
     \> \> \> \> if $v_k$ has a parent $v_\ell$, remove the edge $\{v_k,v_\ell\}$. \\
     \> \>  {\bf end loop} \\
\end{tabbing}
}

The following theorem summarizes the way in which the algorithm will be applied. Its proof is based on a property of matrix congruence known as Sylvester's Law of Inertia, we refer to~\cite{livro} for details.
\begin{theorem}\label{thm:inertia}
Let $M$ be a symmetric matrix of order $n$ that corresponds to a weighted tree $T$ and let $x$ be a real number. Given a bottom-up ordering of $T$, let $D$ be the diagonal matrix produced by Algorithm Diagonalize with entries $T$ and $x$. The following hold:
\begin{itemize}
    \item[(a)] The number of positive entries in the diagonal of $D$ is the number of eigenvalues of $M$ (including multiplicities) that are greater than $-x$.
    \item[(b)] The number of negative entries in the diagonal of $D$ is the number of eigenvalues of $M$ (including multiplicities) that are less than $-x$.
    \item[(c)] The number of zeros in the diagonal of $D$ is the multiplicity of $-x$ as an eigenvalue of $M$.
\end{itemize}
\end{theorem}

\begin{example}\label{ex:sun_diagon}
Consider the star graph $T=K_{1,n}$. In order to locate the eigenvalue $\lambda$ for the deformed Laplacian matrix $M_T(s)$, we use $\mbox{Diagonalize}(M_{T}(s), -\lambda)$ initialized as Figure~\ref{fig:sun_tree_diag}.
\begin{figure}[H]
  \centering
  \includegraphics[width=11cm]{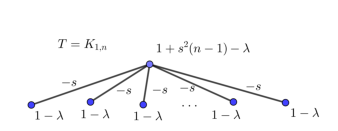}
  \caption{Initial values for the graph $K_{1,n}$.}\label{fig:sun_tree_diag}
\end{figure}
For $\lambda = 1$ we get $n-1$ zeros at the leafs and the pair $(2,-\frac{s^2}{2}$). According to Theorem~\ref{thm:inertia} the spectrum is composed by $n-1$ copies of $\lambda = 1$, an eigenvalue strictly smaller than 1  (so $\lambda = \lambda_{min}$) and other strictly greater than  $\lambda = 1$, the spectral radius (so $\lambda = \lambda_{max}=\rho(M_T(s))$).\\
Applying the algorithm with an eigenvalue $\lambda \not =1 $ we obtain $1-\lambda \neq 0$ in each leaf and the only zero must be the value at the root
$$1+ s^2(n-1) - \lambda -\frac{ns^2}{1 - \lambda} =0$$
whose solutions are
\begin{align*}
 \lambda_{max/min}= \frac{1}{2}\left(s^{2}(n-1)+2\pm s\sqrt{s^{2}(n-1)^{2}+4n}\right)   
\end{align*}
the smaller is the minimum eigenvalue and the other is  $\lambda=\rho(M_T(s))$ (associated to the signal +).
\end{example}

We denote by $d_n(s):=\frac{1} {n}\left(\sum\limits_{i=1}^{n}\lambda_{i}(M_{G}(s))\right)$, the average of the deformed Laplacian eigenvalues of a graph $G$. It is well known that if $G$ is a graph on $n$ vertices and $m$ edges, then $\sum\limits_{i=1}^{n}d_{v_{i}}=2m$. Taking this into account, the following result follows.
\begin{proposition}\label{prop: average eigen}
Let $G$ be a graph on $n$ vertices and $m$ edges. Then,
\begin{align*}
    d_n(s)=\frac{1}{n}\left(\sum\limits_{i=1}^{n}\lambda_{i}(M_{G}(s))\right)=1-s^{2}+\overline{d}s^{2},
\end{align*}
where $\overline{d}=2m/n$ is the average degree of $G$.
\end{proposition}
\begin{proof} \ From \eqref{MG} we have that
    \begin{align*}
    \sum\limits_{i=1}^{n}\lambda_{i}(M_{T}(s)) \ = \ tr(M_{T}(s)) \ &= \ tr(I-sA+s^{2}(D-I)) \\ \\ &= n+s^{2}\sum\limits_{i=1}^{n}(d_{v_{i}}-1) \\ \\ &= n+2ms^{2}-ns^{2} \ = \ (1-s^{2})n+2ms^{2}.
\end{align*}
When we divide the latter by $n$, the result follows.
\end{proof}

The average of the eigenvalues is a parameter of interest in spectral graph theory. For the combinatorial Laplacian (when $s=1$), it is quite important to estimate the number of eigenvalues in certain intervals and is the main theme of several research papers.  We notice that when $G=T$ is a tree, then $m=n-1$ and the general result of Proposition~\ref{prop: average eigen}  becomes
   \begin{equation}\label{eq:dnfortrees}
          d_n(s):=\frac{1}{n}\left(\sum\limits_{i=1}^{n}\lambda_{i}(M_{G}(s))\right)= 1+\frac{n-2}{n}\, s^2,
    \end{equation}
and, in particular, when $s=1$, Equation~\eqref{eq:dnfortrees} is specialized to $2-\frac{2}{n}$. For years, it has been conjectured that most (at least half) of the eigenvalues of any tree of the combinatorial Laplacian are small, that is, smaller than the average $2-\frac{2}{n}$. The proof of this result can be found in \cite{JOT21}.

We note that the formula~\eqref{eq:dnfortrees} is quite interesting since $\frac{n-2}{n}<1$ means $d_n(s) \leq 1+ s^2$, and if $n \to \infty$ then $d_n(s) \to 1+ s^2$. Later, in Proposition~\ref{prop:properties for trees} we will prove that if $T=T'\oplus P_2$ (a tree $T'$ having a pendant path $P_2$) then $\rho(M_T(s))>1+s^2$. Thus it is a very interesting question to predict how many eigenvalues are in the interval $[d_n,  \rho(M_T(s))]$, denoted $\sigma(T)$. We notice that the only possibility for a tree having $\rho(M_T(s))$ close to $1+s^2$ is it to be ``$P_n$ free" for $n\geq 2$, that is, for caterpillars, where the interval $[d_n,  \rho(M_T(s))]$ will be very narrow. Again, by Proposition~\ref{prop:properties for trees} we know that $\rho(M_{T}(s))>1+\sqrt{3}|s|+s^2$ because $\Delta(T)=3$, otherwise $T$ will be a path. So, the length of the interval is at least $\rho(M_T(s))-d_n > \sqrt{3}|s|$. We point out that for $s=\pm 1$ we obtain $d_n=2 - 2/n$ and by \cite{JOT21} we know that $\sigma(T)\leq \lfloor \frac{n}{2}\rfloor$. Those and other related questions would be a subject of additional research.

\section{The spectral radius of a subgraph} \label{sec:The spectral radius of a subgraph}

An expected property for a matrix that represents a graph is that the spectral radius of a subgraph is smaller than the spectral radius of the original graph. We prove that this property holds for the deformed Laplacian matrix, but not for all values of the parameter $s$. In the particular case of trees we will see that we can obtain a stronger claim, see Proposition~\ref{prop:properties for trees}.

\begin{example}\label{ex: not true}
We point out that the deletion of one edge does not necessarily decrease the spectral radius of $M_G(s)$ for $0<s<1$. Consider, for instance, $s=\frac{3}{4}$ and $G$ (see Figure~\ref{fig:graphs})   and $G'$  the graph obtained from $G$ by removing the edge $\{1,2\}$. The adjacency matrix of $G$ is
   \[A=
   \left[ \begin {array}{ccccc} 0&1&1&1&1\\ \noalign{\medskip}1&0&1&0&1
\\ \noalign{\medskip}1&1&0&1&0\\ \noalign{\medskip}1&0&1&0&1
\\ \noalign{\medskip}1&1&0&1&0\end {array} \right]
   \]
\begin{figure}[H]
  \centering
  \includegraphics[width=9cm]{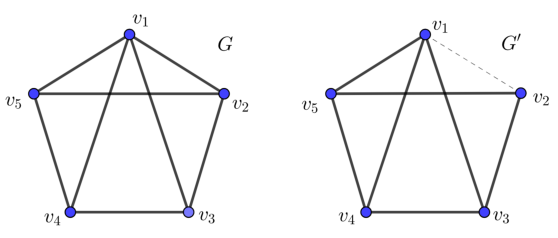}
  \caption{Graphs $G$ (left) and $G'$(right).}\label{fig:graphs}
\end{figure}

  A direct computation shows that $\rho(M_G(s))=3.6250000 < 3.6517332= \rho(M_{G'}(s))$.
\end{example}

Let $G'$ be an undirected simple graph obtained from $G$ by removing the edge $\{v_1, v_2\}$. We now prove that if $s \in \mathbb{R}\setminus (1,0)$, then $\rho(M_G(s)) \geq \rho(M_{G'}(s))$.

In this case the deformed Laplacian matrix is
\begin{align}\label{MGprime}
M_{G'}(s)=
  \left[
    \begin{array}{ccccc}
      1+ s^2 (d_{v_1}-2) & 0 & \cdots & -s a_{1\, n-1} & -s a_{1 \, n} \\
      0 & 1 + s^2 (d_{v_2}-2) & \cdots & -s a_{2\, n-1} & -s a_{2 \,n} \\
      \vdots & \vdots & \ddots & \vdots & \vdots \\
      -s a_{n-1\,1} & -s a_{n-1\,2} & \cdots & 1+ s^2 (d_{v_{n-1}}-1) & -s a_{n-1\,n} \\
      -s a_{n \,1} & -s a_{n\,2} & \cdots & -s a_{n\,n-1} & 1+ s^2 (d_{v_n}-1) \\
    \end{array}
  \right].
\end{align}

Taking into account \eqref{MG} and \eqref{MGprime}, let us define the difference matrix
\begin{align}\label{Hs}
H(s)= M_G(s)-M_{G'}(s)=
  \left[
    \begin{array}{ccccc}
      s^2  & -s & \cdots & 0 & 0 \\
      -s & s^2 & \cdots & 0 & 0 \\
      \vdots & \vdots & \ddots & \vdots & \vdots \\
      0 & 0 & \cdots & 0 & 0 \\
      0 & 0 & \cdots & 0 & 0 \\
    \end{array}
  \right].
\end{align}

The matrix $H$ has some interesting properties, for instance, if $\mathcal{B}=\{e_1, \ldots, e_n\}$ is the canonical basis of $\mathbb{R}^{n}$ then
\[
\left\{
  \begin{array}{ll}
    H e_1= s^2 e_1 - s e_2 \\
     H e_2= - s e_1+ s^2 e_2 \\
    H e_i=0,i=3,\ldots, n
  \end{array}
\right.
\]

As a consequence $H(e_1 + e_2)= (s^2  - s) (e_1 + e_2)$ and $H(e_1 - e_2)= (s^2  + s) (e_1 - e_2)$. In particular $\mathcal{C}=\{u_1, u_2, e_3, \ldots, e_n\}$ is an ortogonal basis consisting of eigenvectors of $H$, where, $u_1 = e_1 + e_2$ is associated with the eigenvalue $\lambda_1=s^2  - s$ and $u_2 = e_1 - e_2$ is associated with the eigenvalue $\lambda_2=s^2 + s$. Some useful relations are
\begin{equation}\label{eq:eq1}
\frac{\lambda_1 + \lambda_2}{2} = s^2 ~~\text{ and } ~~\frac{\lambda_1 - \lambda_2}{2} = -s.
 \end{equation}

We now recall the Rayleigh quotient. If $B$ is an  $n \times n$ real symmetric matrix then
\begin{equation}\label{eq:ray}
\rho(B) = \sup_{x \in \mathbb{R}^{n}/ \{0\}} \frac{\langle x, B x\rangle}{\langle x,  x\rangle} = \sup_{\{x \in \mathbb{R}^{n}\, | \, |x|=1\}} \langle x, B x\rangle.
\end{equation}
\begin{remark}\label{rem:nonposit_H}
Notice that if $|s|> 1$ then $H$ is positive semidefinite, that is, $\langle x, H x\rangle=x^{T}Hx \geq 0$.
\end{remark}

\begin{lemma}(Corollary 8.1.19 \bf{Horn and Johnson})\label{HaJ}
Let $A$ and $B$ be two nonnegative matrices such that $0\leq A\leq B$. Then, $\rho(A)\leq \rho(B)$.
\end{lemma}

\begin{theorem}\label{thm:main_decreasing_rho} Let $G$ be a undirected simple graph containing the edge $\{v_1, v_2\}$ and $G'$ obtained from $G$ by removing the edge $\{v_1, v_2\}$.   If $s \in (-\infty,0] \cup [1, +\infty)$, then $$\rho(M_G(s)) \geq \rho(M_{G'}(s)).$$ 
\end{theorem}
\begin{proof}
From Equation~\eqref{eq:ray}, we have 
\begin{align*}
\rho(M_G(s) ) = \sup_{\{x \in \mathbb{R}^{n}\, | \, |x|=1\}} \left\{\langle x, M_G(s)  x\rangle\right\}
&= \sup_{\{x \in \mathbb{R}^{n}\, | \, |x|=1\}} \left\{\langle x, (M_{G'}(s) +H) x\rangle\right\} \\
&= \sup_{\{x \in \mathbb{R}^{n}\, | \, |x|=1\}} \left\{\langle x, M_{G'}(s) x\rangle + \langle x, H x\rangle\right\}.
\end{align*}
From Remark~\ref{rem:nonposit_H} we know that if $|s|\geq 1$  then $\lambda_1=s^2  - s \geq 0$ and  $\lambda_2=s^2 + s \geq 0$ so the theorem follows for super-Laplacian values of $s$, because $\langle x, H x\rangle \geq 0$ means $\rho(M_G(s)) \geq \rho(M_{G'}(s))$ in the above formula.

So we need to prove only the sub-Laplacian values of $s$, that is $-1<s<0$. In this situation, we always have $\lambda_1=s^2  - s >0$ and  $\lambda_2=s^2 + s<0$. Thus $\langle x, H x\rangle$ can change sign. However, notice that for $-1<s<0$ the matrix defined in \eqref{Hs} satisfies
$$H(s)= M_G(s)-M_{G'}(s)\geq 0.$$
Also, in \eqref{MGprime}, we have $0\leq M_{G'}(s)$. Consequently, $0\leq M_{G'}(s)\leq M_{G}(s)$. Therefore, from Lemma \ref{HaJ}, it follows that $\rho(M_{G'}(s)) \leq \rho(M_{G}(s))$.
\end{proof}

From Theorem~\ref{thm:main_decreasing_rho} we can conclude the following result:

\begin{corollary}\label{DET} Let $G$ be an undirected simple graph and let $G'$ any subgraph of $G$. Then, $\rho(M_G(s)) \geq \rho(M_{G'}(s))$, for all $s \in (-\infty,0] \cup [1, +\infty)$.
\end{corollary}

\section{Bounds for the spectral radius of the deformed Laplacian}\label{sec:Some bounds on the spectral radius of the deformed Laplacian}
We start with the following lower bound.
\begin{theorem}
Let $G$ be a simple, undirected and connected graph of order $n$ with maximum degree $\Delta$. If $s\in (-\infty, 0]\cup [1,+\infty)$, then
\begin{align*}
\rho(M_{G}(s))\geq \frac{1}{2}\left(s^{2}(\Delta-1)+2+\vert s\vert\sqrt{s^{2}(\Delta-1)^{2}+4\Delta}\right).
\end{align*}
The equality holds if and only if $G\cong K_{1,n-1}$.
\end{theorem}
\begin{proof}
It is known that the deformed Laplacian spectrum of a star (see Example~\ref{ex:sun_diagon}) of order $n$ is given by:
\begin{align}\label{I1}
\sigma(K_{1,n})=\left\{\frac{1}{2}\left(s^{2}(n-1)+2\pm s\sqrt{s^{2}(n-1)^{2}+4n}\right), \ 1^{[n-1]}\right\}.
\end{align}
Since $G$ has maximum degree $\Delta$, $K_{1,\Delta}$ is a proper subgraph of $G$. From Corollary \ref{DET} and \eqref{I1}, it follows that
\begin{align*}
	\rho(M_{G}(s)) \geq \rho(M_{K_{1,\Delta}}(s)) = \frac{1}{2}\left(s^{2}(\Delta-1)+2+\vert s\vert\sqrt{s^{2}(\Delta-1)^{2}+4\Delta}\right).
\end{align*}
\end{proof}

We next provide the following upper bound. \begin{theorem}\label{the:bound}
Let $G$ be a simple, undirected and connected graph of order $n$ with maximum degree $\Delta$. Let $M_G(s)$ be its deformed Laplacian matrix. Then,
	\begin{align*}
		\rho(M_{G}(s))\leq 1+s^2(\Delta-1)+|s|\rho(A(G)).
	\end{align*}
\end{theorem}
\begin{proof} Recall that $M_{G}(s)=I-sA+s^2(D-I)$. Thus,
\begin{align*}
	\rho(M_{G}(s))\leq \sup_{|x|=1}\left\{|M_{G}(s)\,x|\right\} &= \sup_{|x|=1} |(I-sA+s^2(D-I))\,x| \\
	&\leq \sup_{|x|=1} |(I+s^2(D-I))\,x|+\sup_{|x|=1}|s||A\,x| \\
	&= 1+s^2(\Delta-1) +|s|\rho(A(G)),
\end{align*}
since $I+s^2(D-I)$ is a diagonal matrix and $\sup\limits_{|x|=1}|A\,x| = \rho(A(G))$, being $A$ a symmetric matrix.

\end{proof}

\section{Spectrum of deformed Laplacian associated to a tree}\label{sec:Spectrum of deformed Laplacian associated to a tree}

We start this section, with a result for the spectrum of the deformed Laplacian associated to a bipartite graph, which contains trees as a particular subcase.

\begin{theorem}\label{th:bip}
	If $G$ is a bipartite graph, then $\sigma(M_{G}(s)) =\sigma(M_{G}(-s))$.
\end{theorem}
\begin{proof} For $G=(V,E)$, let $V=V_1 \cup V_2$ the bipartition of $V$. We write
	$$A=A(G)= \begin{pmatrix}
		0 &  & B \\
		&  &  \\
		B^T &  & 0 \\
	\end{pmatrix}$$
	Let $U$ be the diagonal matrix defined as follows:
	\begin{align*}
		u_{i,i}=\left\{\begin{array}{lr}
			1 & \ \text{ if } \ v_i \in V_1 \\ \\
			-1 & \ \text{ if } \ v_i \in V_2.
		\end{array}\right.
	\end{align*}\ \\
It is easy to see that $U= U^{-1}$ and also $UAU^{-1}=-A$. Let us now write
\begin{align*}
M_1=M_{G}(s)=I+s^2(D-I)-sA \qquad \text{and} \qquad M_2=M_{G}(-s)=I+s^2(D-I)+sA.
\end{align*}
Then,
	\begin{align*}
		UM_{1}U^{-1} &= U(I+s^2(D-I)-sA)U^{-1} \\
		&=UU^{-1}+s^{2}UU^{-1}(D-I)-sUAU^{-1}\\
		&=I+s^2(D-I)-s(-A)=M_2.
	\end{align*}
It follows then that $M_1$ and $M_2$ are similar matrices, and therefore have the same spectrum.
\end{proof}

\begin{remark}
{\rm It is well known in the literature that when $G$ is a bipartite graph, the Laplacian matrix $L(G)$ and the signless Laplacian matrix $Q(G)$ are cospectral. This can be seen as a particular case of Theorem~\ref{th:bip} as follows:
\begin{align*}
	\sigma(Q(G))=\sigma(M_{G}(1))=\sigma(M_{G}(-1))=\sigma(L(G)).
\end{align*}}
\end{remark}

We now give some results about the spectrum of the deformed Laplacian matrix $M_T(s)$ underlying a tree $T$.  In particular, the Proposition~\ref{prop:properties for trees} 5) strengths the result in Theorem~\ref{thm:main_decreasing_rho}, for general graphs, showing that for every value of $s$ the spectral radius of sub trees does not increase.

\begin{proposition}\label{prop:properties for trees}
Let $T\neq K_1$ be a tree and let $M_T(s)$ be its deformed Laplacian matrix. Then
\begin{itemize}
    \item[1)] $0\in Spect(M_T(s))$ if and only if $s=\pm 1$;
    \item[2)] $M_T(s)$ is positive definite if and only if $|s|<1$;
    \item[3)] $\rho(M_T(s))>1$;
    \item[4)] If $T$ has a pendant $P_2$ then $\rho(M_T(s))>1+s^2$;
    \item[5)] Let $v$ be a pendant vertex of a tree $T$. If $T'=T\setminus v$ then $\rho(M_{T'}(s))<\rho(M_{T}(s))$. In particular, if $T'$ is a subtree of $T$, then $\rho(M_{T'}(s))<\rho(M_{T}(s))$;
    \item[6)] If $\Delta(T)\geq 4$, then $\rho(M_{T}(s))>1+2|s|+s^2$ (analogously, $\rho(M_{T}(s))>1+\sqrt{3}|s|+s^2$ if $\Delta(T)\geq 3$);
    \item[7)] If $T$ is a starlike tree $T:=[q_1,\ldots,q_k]$ then $\rho(M_{T}(s))\leq 1+s^2(\Delta(T)-1) +|s|\frac{k}{\sqrt{k-1}}$.
\end{itemize}
\end{proposition}
\begin{proof}
\begin{enumerate}
    \item  We perform the algorithm $Diag(M_T(s),0)$. Initializing $T$ with the corresponding values we see that in each pendant star with central vertex $u$ we obtain the value $1+s^2(\deg(u)-1)$, $-s$ in the edges and $1$ in each leaf. The first stage updates the value in $u$ as \[ 1+s^2(\deg(u)-1)-(\deg(u)-1)\frac{(-s)^2}{1}=1. \]
    At the last vertex (the root) we obtain \[ 1+s^2\deg(u)-(\deg(u)-1)\frac{(-s)^2}{1}=1-s^2. \] Thus, $0$ is an eigenvalue of $M_T(s)$ if and only if $s=\pm 1$ and all the other eigenvalues of the deformed Laplacian are greaterr than $0$.
    \item  Observe that if $|s|<1$, then all output values of the algorithm are positive, which means that all eigenvalues are greater than $0$.
    \item  Consider $\lambda=\rho(M_{T}(s))$. When performing the algorithm $Diag(M_T(s),-\lambda)$, each leaf is initialized with $1-\lambda\leq 0$ because we are using the spectral radius and by Theorem~\ref{thm:inertia} all the values must be negative. Moreover, if $1-\lambda=0$ and since $T\neq K_1$, then that leaf should receive the value $2$ and the next vertex receives $\frac{-s^2}{2}$. However, a positive output means that there is an eigenvalue greater than $\lambda$, which is absurd since $\lambda$ is the spectral radius. Thus $1-\lambda<0$ or $\lambda>1$.
    \item  Consider $T=T'\oplus P_2$ (a tree with a pendant $p_2$ and $\lambda=\rho(M_{T}(s))$. When performing the algorithm $Diag(M_T(s),-\lambda)$, each leaf is initialized with $1-\lambda<0$, as we already proved. The next vertex in $P_2$ receives the value \[ 1+s^2(2-1)-\lambda-\frac{s^2}{1-\lambda}<0, \] otherwise we reach a contradiction. Since $1-\lambda<0$, we obtain \[ 1+s^2-\lambda<\frac{s^2}{1-\lambda}<0, \] thus $\rho(M_{T}(s))>1+s^2$.
    \item  Consider $T'=T\setminus v$, where $T'$ has the joining vertex $u$ as the root of $T'$. We know that for $\lambda=\rho(M_{T'}(s))$ we obtain \[ a_{T'}(u)=1+s^2(\deg_{T'}(u)-1)-\lambda-\xi=0 \] where $\xi$ is the resulting value obtained by processing all the children of $u$ in $T'$. On the other hand, \[ a_T(u)=1+s^2(\deg_T(u)-1)-\lambda-\xi-\frac{s^2}{1-\lambda} = a_{T'}(u)+s^2+\frac{s^2}{\lambda-1}>0, \] thus $\rho(M_{T}(s))>\rho(M_{T'}(s))$.
    \item  If $\Delta(T)\geq 4$, then $S_4$ is a subtree of $T$ and therefore $\rho(M_{S_4}(s))<\rho(M_{T}(s))$. Let $u$ be the root of $S_4$ and $\lambda=\rho(M_{S_4}(s))$. Then \[a_{S_4}(u)=1+s^2(4-1)-\lambda-\frac{4s^2}{1-\lambda}=0. \] Solving this equation for $\lambda$ we get one positive root, given by \[ \lambda=\frac{3}{2}s^2+1+\frac{1}{2}\sqrt{9s^4+16s^2}>1+s^2+\frac{1}{2}4|s|\sqrt{\frac{9}{16}s^2+1}>1+s^2+2|s|. \]
    \item If $T$ is a starlike tree $T:=[q_1,\ldots, q_k]$, then we have from \cite{starlike} that \[ \sqrt{k} \leq \rho_A(T)\leq \frac{k}{\sqrt{k-1}}. \] The inequality follows from Theorem~\ref{the:bound}.
\end{enumerate}
\end{proof}

\begin{example}\label{ex:path}
Consider the tree $T=P_n$ and its deformed Laplacian matrix. Assume that $\lambda \geq \rho(M_{P_n}(s))$.
\begin{figure}[H]
  \centering
  \includegraphics[width=11cm]{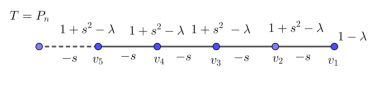}
  \caption{Graph $P_{n}$.}\label{fig:weighted_path_diag}
\end{figure}
The output of the Eigenvalue Location Algorithm ($\mbox{Diag}(M_{P_n}(s), -\lambda)$) at each vertex of a path $P_n$ (see Figure~\ref{fig:weighted_path_diag}) gives a sequence $Z_1, Z_2, \ldots, Z_n$ defined by the recurrence relation \begin{align*}
    Z_1 & =1-\lambda \\
    Z_j & =\varphi(Z_{j-1}), \text{ for } 2 \leq j \leq n-1,
\end{align*}
where $\varphi(t)=1+s^2-\lambda -\frac{s^2}{t}$, for $t\neq 0$.  Although, the last value processed in $v_n$ is different because it has degree 1,
$$Z_{n} =  1-\lambda - \frac{s^2}{Z_{n-1}}.$$
This equation provides an explicit formula for the spectral radius of the deformed Laplacian matrix associated to a path (and for starlike graphs) since $Z_{j}$ can be explicitly computed from \cite[Theorem 1]{OliveTrevAppDiff} and the equation $Z_{n} = 0$ will be a simple rational expression on $\lambda = \rho(M_{P_n}(s))$ (the denominator is always non zero because all the previous $Z_j$ must be negative).\\
The fixed points of $\varphi$, that is, the solutions of $\varphi(t)=t$ are the solutions $\theta(s)$ and $\theta'(s)$ of the quadratic equation
\begin{equation}
    t^2-(1+s^2-\lambda)t+s^2=0.
\end{equation}
The most interesting case is when this equation has two distinct solutions, which happens when \begin{equation}\label{ineq}
    \Delta(s):=(1+s^2-\lambda)^2-4s^2>0
\end{equation} or, equivalently, $1-\sqrt{\lambda}<s<-1+\sqrt{\lambda}$. The solutions, in this case, are given by
\begin{align}
     \theta(s) & :=\frac{(1+s^2-\lambda)-\sqrt{(1+s^2-\lambda)^2-4s^2}}{2} \quad \text{ and } \\
    \theta'(s) & :=\frac{(1+s^2-\lambda)+\sqrt{(1+s^2-\lambda)^2-4s^2}}{2}.
\end{align}
It is obvious that $Z_1=1-\lambda <0$ if $\lambda \geq \rho(M_{P_n}(s))$ because $\rho(M_{P_n}(s))>1$, from Proposition~\ref{prop:properties for trees}, (3). In this case, $Z_1<\theta$.
Indeed, we have $1-\lambda<\theta$ if and only if
    \[ 2-2\lambda < 1+s^2-\lambda-\sqrt{(1+s^2-\lambda)^2-4s^2} \]
    if and only if
\[ \sqrt{(1+s^2-\lambda)^2-4s^2}<-1+s^2+\lambda,\] which happens if and only if $-2\lambda s^2<2\lambda s^2$, which is true since $\lambda>1$.
Recall that, from Definition~\ref{def:adapted} in both sub-Laplacian or super-Laplacian cases, given a real number $\lambda>1$, if $s$ is adapted to $\lambda$ then $\lambda>(1+|s|)^2=1+ 2|s| +|s|^2$ so $1-\sqrt{\lambda}<s<-1+\sqrt{\lambda}$.  This fact, combined with our estimates of the spectral radius of trees with minimum greater than 4 (Proposition~\ref{prop:properties for trees}, (6)), allows us to understand the behaviour of the deformed Laplacian eigenvalues associated to paths. More specifically, we observe that all $Z_j$ are smaller than $\theta(s)$ because $Z_j \to \theta(s)$, increasingly, when $n \to \infty$ (see \cite[Section 3.3]{OliveTrevAppDiff} for details). This result can be used in several applications as consensus protocol or asymptotic behaviour of networks (that is, when we increase the number of vertices).
\end{example}

\section{Spectrum of the deformed Laplacian associated to an \emph{H-join} graph}\label{sec:Spectrum of the deformed Laplacian associated to a H-join graph}
Let $H$ be a graph of order $r$ and let $G=\bigvee\limits_{H}{\{G_{i}:1\leq i\leq r\}}$ be the \emph{H-join} graph. In this section, we consider the vertices of $G$ with the labels $1,\ldots ,\sum_{i=1}^{r}n_{i}$, starting with the vertices of $G_{1}$, continuing with the vertices of $G_{2},G_{3},\ldots ,G_{r-1}$ and finally with the vertices of $G_{r}.$
\begin{example}\label{ExDOT}
Let $H=P_{3}$, $G_{1}=C_{4}$, $G_{2}=P_{2}$, $G_{3}=C_{6}$. The graph $G=\bigvee\limits_{H}\{G_{1},G_{2},G_{3}\}$ is depicted in Figure 4, with vertex labels as described.
\begin{eqnarray*}
& \definecolor{qqqqff}{rgb}{0.,0.,1.}
\definecolor{cqcqcq}{rgb}{0.752941176471,0.752941176471,0.752941176471}
\begin{tikzpicture}
\draw [line width=1.2pt] (7.,1.)-- (7.,-1.);
\draw [line width=1.2pt] (7.,-1.)-- (5.5,-2.);
\draw [line width=1.2pt] (5.5,-2.)-- (4.,-1.);
\draw [line width=1.2pt] (4.,-1.)-- (4.,1.);
\draw [line width=1.2pt] (2.,1.5)-- (2.,-1.5);
\draw [line width=1.2pt] (-1.5,2.)-- (-3.,0.);
\draw [line width=1.2pt] (-3.,0.)-- (-1.5,-2.);
\draw [line width=1.2pt] (0.,0.)-- (-1.5,-2.);
\draw [line width=1.2pt] (-1.5,2.)-- (0.,0.);
\draw [line width=1.2pt] (7.,1.)-- (5.5,2.);
\draw [line width=1.2pt] (5.5,2.)-- (4.,1.);
\draw [line width=0.4pt] (2.,1.5)-- (4.,1.);
\draw [line width=0.4pt] (2.,1.5)-- (5.5,2.);
\draw [line width=0.4pt] (2.,1.5)-- (7.,1.);
\draw [line width=0.4pt] (2.,1.5)-- (4.,-1.);
\draw [line width=0.4pt] (2.,1.5)-- (5.5,-2.);
\draw [line width=0.4pt] (2.,1.5)-- (7.,-1.);
\draw [line width=0.4pt] (2.,-1.5)-- (4.,1.);
\draw [line width=0.4pt] (2.,-1.5)-- (5.5,2.);
\draw [line width=0.4pt] (2.,-1.5)-- (7.,1.);
\draw [line width=0.4pt] (2.,-1.5)-- (5.5,-2.);
\draw [line width=0.4pt] (2.,-1.5)-- (4.,-1.);
\draw [line width=0.4pt] (2.,-1.5)-- (7.,-1.);
\draw [line width=0.4pt] (0.,0.)-- (2.,1.5);
\draw [line width=0.4pt] (0.,0.)-- (2.,-1.5);
\draw [line width=0.4pt] (-1.5,-2.)-- (2.,-1.5);
\draw [line width=0.4pt] (-1.5,-2.)-- (2.,1.5);
\draw [line width=0.4pt] (-3.,0.)-- (2.,1.5);
\draw [line width=0.4pt] (-3.,0.)-- (2.,-1.5);
\draw [line width=0.4pt] (-1.5,2.)-- (2.,1.5);
\draw [line width=0.4pt] (-1.5,2.)-- (2.,-1.5);
\begin{scriptsize}
\draw [fill=qqqqff] (4.,1.) circle (2.0pt);
\draw[color=qqqqff] (4.2,0.9) node {$8$};
\draw [fill=qqqqff] (7.,-1.) circle (2.0pt);
\draw[color=qqqqff] (7.3,-0.9) node {$11$};
\draw [fill=qqqqff] (4.,-1.) circle (2.0pt);
\draw[color=qqqqff] (4.2,-0.9) node {$9$};
\draw [fill=qqqqff] (5.5,-2.) circle (2.0pt);
\draw[color=qqqqff] (5.5,-2.3) node {$10$};
\draw [fill=qqqqff] (2.,1.5) circle (2.0pt);
\draw[color=qqqqff] (2.,1.8) node {$5$};
\draw [fill=qqqqff] (2.,-1.5) circle (2.0pt);
\draw[color=qqqqff] (2.,-1.8) node {$6$};
\draw [fill=qqqqff] (-1.5,2.) circle (2.0pt);
\draw[color=qqqqff] (-1.5,2.3) node {$1$};
\draw [fill=qqqqff] (-3.,0.) circle (2.0pt);
\draw[color=qqqqff] (-3.3,0.) node {$2$};
\draw [fill=qqqqff] (0.,0.) circle (2.0pt);
\draw[color=qqqqff] (-0.3,0.) node {$4$};
\draw [fill=qqqqff] (-1.5,-2.) circle (2.0pt);
\draw[color=qqqqff] (-1.5,-2.3) node {$3$};
\draw [fill=qqqqff] (5.5,2.) circle (2.0pt);
\draw[color=qqqqff] (5.5,2.3) node {$7$};
\draw [fill=qqqqff] (7.,1.) circle (2.0pt);
\draw[color=qqqqff] (7.3,0.9) node {$12$};
\end{scriptsize}
\end{tikzpicture}&\\
&\text{{\bf Figure 4}. The graph $G=\bigvee\limits_{P_3}\{C_{4}, P_{2}, C_{6}\}$.}&
\end{eqnarray*}
\end{example}

With the above mentioned labeling, let
\begin{equation}
M\left( G\right) = \left[
\begin{array}{cccc}
M_{1} & \delta _{12}\mathbf{1}_{n_{1}}\mathbf{1}_{n_{2}}^{T} & \ldots  &
\delta _{1r}\mathbf{1}_{n_{1}}\mathbf{1}_{n_{r}}^{T} \\
\delta _{12}\mathbf{1}_{n_{2}}\mathbf{1}_{n_{1}}^{T} & M_{2} & \ddots  &
\vdots  \\
\vdots  & \ddots  & \ddots  & \delta _{\left( r-1\right) r}\mathbf{1}_{n_{r-1}}\mathbf{1}_{n_{r}}^{T} \\
\delta _{1r}\mathbf{1}_{n_{r}}\mathbf{1}_{n_{1}}^{T} & \ldots  & \delta_{\left( r-1\right) r}\mathbf{1}_{n_{r}}\mathbf{1}_{n_{r-1}}^{T} & M_{r}%
\end{array}%
\right],   \label{R}
\end{equation}%
where the diagonal blocks $M_{i}$ are symmetric matrices such that
\begin{equation*}
M_{i}\mathbf{1}_{n_{i}}=\mu _{i}\mathbf{1}_{n_{i}}
\end{equation*}
for $i=1,\ldots ,r$ and $\delta_{i,j}$ are scalars for $1 \leq i < j\leq r $. Let
\begin{align}
F_{r}=\left[\begin{array}{ccccc}
      \mu_{1}                           & \delta_{12}\sqrt{n_{1}n_{2}}      & \ldots & \delta_{1(r-1)}\sqrt{n_{1}n_{r-1}}           & \delta_{1r}\sqrt{n_{1}n_{r}} \\
      \delta_{12}\sqrt{n_{1}n_{2}}      & \mu _{2}                          & \ldots & \delta_{2(r-1)}\sqrt{n_{2}n_{r-1}}           & \delta_{2r}\sqrt{n_{2}n_{r}} \\
      \vdots                            & \vdots                            & \ddots & \vdots                                       & \vdots \\
      \delta_{1(r-1)}\sqrt{n_{2}n_{r-1}}& \delta_{2(r-1)}\sqrt{n_{2}n_{r-1}}& \ldots & \mu_{r-1}                                    & \delta_{\left(r-1\right)r}\sqrt{n_{r-1}n_{r}}\\
      \delta_{1r}\sqrt{n_{1}n_{r}}      & \delta_{2r}\sqrt{n_{2}n_{r}}      & \ldots & \delta_{\left(r-1\right)r}\sqrt{n_{r-1}n_{r}}& \mu_{r}
      \end{array}\right].  \label{C}
\end{align}\ \\

The authors in \cite{CDR} obtain the spectrum of $M(G)$, when $G=\bigvee\limits_{H}{\{G_{i}:1\leq i\leq r\}}$ as follows:

\begin{theorem}\label{H-Join_distance_eigenvalues}
Consider the matrix $M(G)$ \eqref{R} and the matrix $F_r$ \eqref{C}. Then the spectrum of $M\left(G\right)$ is
\begin{equation*}
\sigma(M(G))=\bigcup _{i=1}^{r}\big(\sigma(M_{i})-\{\mu _{i}\}\big) \cup \sigma(F_{r}).
\end{equation*}
\end{theorem}

\subsection{The deformed Laplacian spectrum for H-join of graphs}\label{SDLspect}
In this section, we assume that $H$ is a connected graph of order $r$. From \eqref{R}, we have that the deformed Laplacian matrix of the \emph{H-join} $G=\bigvee\limits_{H}{\{G_{i}:1\leq i\leq r\}}$, where $G_i$
is a $d_{i}$-degree regular graph of order $n_{i}$, for each $i=1,\ldots,r$, \ is given as follows
\begin{align}\label{D}
M_{G}(s)=I-sA+s^{2}(D-I)=\left[
\begin{array}{cccc}
M_{1}(s) & \delta _{12}\mathbf{1}_{n_{1}}\mathbf{1}_{n_{2}}^{T} & \ldots  &
\delta _{1r}\mathbf{1}_{n_{1}}\mathbf{1}_{n_{r}}^{T} \\
\delta _{12}\mathbf{1}_{n_{2}}\mathbf{1}_{n_{1}}^{T} & M_{2}(s) & \ddots  &
\vdots  \\
\vdots  & \ddots  & \ddots  & \delta _{\left( r-1\right) r}\mathbf{1}%
_{n_{r-1}}\mathbf{1}_{n_{r}}^{T} \\
\delta _{1r}\mathbf{1}_{n_{r}}\mathbf{1}_{n_{1}}^{T} & \ldots  & \delta
_{\left( r-1\right) r}\mathbf{1}_{n_{r}}\mathbf{1}_{n_{r-1}}^{T} & M_{r}(s)
\end{array}%
\right],
\end{align}
and from (\ref{C})
\begin{equation}\label{FD}
F_{r}(s)=\left[
\begin{array}{cccc}
\lambda_{1}(M_{1}(s)) & \delta _{12}\sqrt{n_{1}n_{2}} & \ldots & \delta _{1r}\sqrt{n_{1}n_{r}} \\
\delta _{12}\sqrt{n_{1}n_{2}} & \lambda_{1}(M_{2}(s)) & \ddots & \vdots \\
\vdots & \ddots & \ddots & \delta _{\left( r-1\right) r}\sqrt{n_{r-1}n_{r}}
\\
\delta _{1r}\sqrt{n_{1}n_{r}} & \ldots & \delta _{\left( r-1\right) r}\sqrt{%
n_{r-1}n_{r}} & \lambda_{1}(M_{r}(s))%
\end{array}%
\right],
\end{equation}
where $\delta_{ij}=-s$ if $ij\in E(H)$ and $\delta_{ij}=0$ otherwise.
\begin{remark}\label{N1}
Let $N_{i}=\sum\limits_{j\in N_{H}(i)}n_{j}$, where $N_{H}(i)$ is the set of neighbors of vertice $i\in V(H)$. Taking into account \eqref{D}, we note that
\begin{align}\label{Mimatrix}
M_{i}(s)=\left(s^{2}(d_{i}+N_{i}-1)+1\right)I_{n_{i}}-sA(G_{i}),
\end{align}
where $I_{n_{i}}$ is the identity matrix of order $n_{i}$ and $A(G_{i})$ is the adjacency matrix of $G_{i}$, for each $i=1,\ldots,r$.
Further, as $G_i$ is a $d_{i}$-degree regular graph it follows that $M_{i}(s)\mathbf{1}_{n_{i}} = \lambda_{1}(M_{i}(s))\mathbf{1}_{n_{i}}$, where
the eigenvalues $\lambda_{1}(M_{i}(s))$ (in \eqref{FD}) are given by
\begin{align}\label{MaxEigMi}
\lambda_{1}(M_{i}(s))=s^{2}(d_{i}+N_{i}-1)-s d_{i}+1, \qquad i=1,\ldots,r.
\end{align}
\noindent
We conclude that the spectrum of $M_{i}(s)$ is completely determined by the spectrum of $A(G_{i})$, and is given as follows:
\begin{align}\label{SM}
\sigma(M_{i}(s))=\left\{s^{2}(d_{i}+N_{i}-1)-s\lambda_{k}(A(G_{i}))+1\right\}_{k=1}^{n_{i}} \ ,
\end{align}
with $\lambda_{1}(A(G_{i}))=d_{i}$, for each $i=1,\ldots,r$.
\end{remark}
Taking into account the Remark~\ref{N1}, by applying Theorem~\ref{H-Join_distance_eigenvalues}, we obtain our main result:
\begin{theorem}\label{TD}
Let $H$ be a connected graph of order $r$. If for each $i \in \{1, \ldots, r\}$, $G_i$ is a $d_{i}$-degree regular graph of order $n_{i}$, then the spectrum of $M_{G}(s)$,
where $G=\bigvee\limits_{H}{\{G_{i}:1\leq i\leq r\}}$, is given by
\begin{equation*}
\sigma(M_{G}(s))=\bigcup _{i=1}^{r}\bigg(\sigma(M_{i}(s))-\{\lambda_{1}(M_{i}(s))\}\bigg) \cup \sigma(F_{r}(s))
\end{equation*}%
where $F_{r}(s)$, $\lambda_{1}(M_{i}(s))$ and $\sigma(M_{i}(s))$ are given as in \eqref{FD}, \eqref{MaxEigMi} and \eqref{SM}, respectively.
\end{theorem}

When $H=P_{r}$ or $H=C_{r}$, Theorem \ref{TD} allows us to obtain the following results:
\begin{corollary}\label{CorCam}
	Let $H=P_{r}$. If for each $i \in \{1, \ldots, r\}$, $G_i$ is a $d_{i}$-degree regular graph of order $n_{i}$, then the spectrum of $M_{G}(s)$,
	where $G=\bigvee\limits_{P_{r}}{\{G_{i}:1\leq i\leq r\}}$, is given by
	\begin{equation*}
		\sigma(M_{G}(s))=\bigcup _{i=1}^{r}\bigg(\sigma(M_{i}(s))-\{\lambda_{1}(M_{i}(s))\}\bigg) \cup \sigma(F_{r}(s))
	\end{equation*}%
	where $\lambda_{1}(M_{i}(s))$ and $\sigma(M_{i}(s))$ are given as in \eqref{MaxEigMi} and \eqref{SM}, respectively and $F_{r}(s)$ is the tridiagonal matrix given by
\begin{align*}
	F_{r}(s)=\begin{bmatrix}
		\lambda_{1}(M_{1}(s)) & -s\sqrt{n_{1}n_{2}} \\
		-s\sqrt{n_{1}n_{2}} & \lambda_{1}(M_{2}(s)) & -s\sqrt{n_{2}n_{3}} \\
		& -s\sqrt{n_{2}n_{3}} & \lambda_{1}(M_{3}(s)) & -s\sqrt{n_{3}n_{4}} \\
		& & -s\sqrt{n_{3}n_{4}} & \ddots & \ddots \\
		& & & \ddots & \ddots &
		\\
		& & & & \lambda_{1}(M_{r-1}(s)) & -s\sqrt{n_{r-1}n_{r}} \\
		& & & & -s\sqrt{n_{r-1}n_{r}} & \lambda_{1}(M_{r}(s))
	\end{bmatrix}.
\end{align*}
\end{corollary}

\begin{corollary}\label{CorCic}
	Let $H=C_{r}$. If for each $i \in \{1, \ldots, r\}$, $G_i$ is a $d_{i}$-degree regular graph, then the spectrum of $M_{G}(s)$,
	where $G=\bigvee\limits_{C_{r}}{\{G_{i}:1\leq i\leq r\}}$, is given by
	\begin{equation*}
		\sigma(M_{G}(s))=\bigcup _{i=1}^{r}\bigg(\sigma(M_{i}(s))-\{\lambda_{1}(M_{i}(s))\}\bigg) \cup \sigma(F_{r}(s))
	\end{equation*}%
	where $\lambda_{1}(M_{i}(s))$ and $\sigma(M_{i}(s))$ are given as in \eqref{MaxEigMi} and \eqref{SM}, respectively and $F_{r}(s)$ is the periodic Jacobi matrix given by
	\begin{align*}
		F_{r}(s)=\begin{bmatrix}
			\lambda_{1}(M_{1}(s)) & -s\sqrt{n_{1}n_{2}} & & & & -s\sqrt{n_{1}n_{r}}  \\
			-s\sqrt{n_{1}n_{2}} & \lambda_{1}(M_{2}(s)) & -s\sqrt{n_{2}n_{3}} \\
			& -s\sqrt{n_{2}n_{3}} & \lambda_{1}(M_{3}(s)) & -s\sqrt{n_{3}n_{4}} \\
			& & -s\sqrt{n_{3}n_{4}} & \ddots & \ddots \\
			& & & \ddots & \ddots &
			\\
			& & & & \lambda_{1}(M_{r-1}(s)) & -s\sqrt{n_{r-1}n_{r}} \\
			-s\sqrt{n_{1}n_{r}} & & & & -s\sqrt{n_{r-1}n_{r}} & \lambda_{1}(M_{r}(s))
		\end{bmatrix}.
	\end{align*}
\end{corollary}

\subsubsection{Some particular cases of Corollary \ref{CorCam}}
It is well known that the determinant of a symmetric tridiagonal matrix $A$ of order $n$ satisfies the following recurrence relation:
\begin{align}\label{det}
f_{n}=\vert A\vert = \begin{vmatrix}
a_1 & b_1 & 0 & \cdots & 0 \\
b_1 & a_2 & b_2 & \ddots & \vdots \\
0 & b_2 & a_3 & \ddots & 0 \\
\vdots & \ddots & \ddots & \ddots & b_{n-1} \\
0 & \cdots & 0 & b_{n-1} & a_n
\end{vmatrix} = a_{n}f_{n-1}-b^{2}_{n-1}f_{n-2}  , 
\end{align}
where $f_{n-1}$ and $f_{n-2}$ are the determinants of the leading principal submatrices of $A$ of order $n-1$ and $n-2$, respectively, with $f_{0}=1$.

\begin{Res}\label{ResRD}
Let $H=P_{3}$. Let $G_{i}$ be a $d_{i}-$regular graph of order $n_{i}$, $i=1,2,3$. If $d_{1}=d_{3}$ and $G=\bigvee\limits_{P_{3}}{\{G_{1}, G_{2}, G_{3}\}}$, then
		\begin{equation*}
			\sigma(M_{G}(s))=\bigcup _{i=1}^{3}\sigma(M_{i}(s)) \cup \left\{\frac{1}{2}\left(a(s)+b(s)\pm \sqrt{\left(a(s)-b(s)\right)^{2}+4s^{2}n_{2}(n_{1}+n_{3})}\right)\right\}-\{a(s), \ b(s)\},
	\end{equation*}
\noindent
where 
\begin{align*}
 a(s)&=\lambda_{1}(M_{1}(s))=s^{2}(d_{1}+n_{2}-1)-s d_{1}+1, \\ \\ b(s)&=\lambda_{1}(M_{2}(s))=s^{2}(d_{2}+(n_{1}+n_{2})-1)-s d_{2}+1   
\end{align*}
and $\sigma(M_{i}(s))$, $i=1,2,3$, is given as in \eqref{SM}.
\end{Res}
\begin{proof}
Since $d_{1}=d_{3}$, from \eqref{MaxEigMi}, it follows that $a(s)=\lambda_{1}(M_{1}(s))=\lambda_{1}(M_{3}(s))$. Therefore,
\begin{align*}
	F_{3}(s)=\begin{bmatrix}
		a(s) & -s\sqrt{n_{1}n_{2}} & 0 \\
		-s\sqrt{n_{1}n_{2}} & b(s) & -s\sqrt{n_{2}n_{3}} \\
		0 & -s\sqrt{n_{2}n_{3}} & a(s)
	\end{bmatrix}.
\end{align*}
From \eqref{det}, we have
\begin{align*}
\vert \lambda I-F_{3}(s)\vert &=(\lambda-a(s))\begin{vmatrix}
		a(s) & -s\sqrt{n_{1}n_{2}} \\
		-s\sqrt{n_{1}n_{2}} & b(s)
        \end{vmatrix}-s^{2}n_{2}n_{3}(\lambda-a(s)) \\ \\ &=(\lambda-a(s))\left(\lambda^{2}-(a(s)+b(s))\lambda+a(s)b(s)-s^{2}n_{2}(n_{1}+n_{3})\right).
\end{align*}
Consequently,
\begin{align*}
\sigma(F_{3}(s))=\left\{a(s), \ \ \frac{1}{2}\left(a(s)+b(s)\pm \sqrt{\left(a(s)-b(s)\right)^{2}+4s^{2}n_{2}(n_{1}+n_{3})}\right)\right\}.
\end{align*}
Thus, from Corollary \ref{CorCam}, the proof follows.
\end{proof}

Next, we illustrate Result \ref{ResRD} with the following example:
\begin{example}\label{ExRD}
	Let $G=\bigvee\limits_{P_3}\{C_{4}, P_{2}, C_{6}\}$ be the graph depicted in Example \ref{ExDOT}. Taking into account \eqref{SM} and Result \ref{ResRD}, we explicitly obtain the spectrum of $M_{G}(s)$ as follows:	
    \begin{multline*}
		\sigma(M_{G}(s))=\bigg\{\frac{1}{2}\left(13s^{2}-3s+2\pm s\sqrt{49s^{2}+14s+81}\right) \ , \ 3s^{2}-2s+1 \ , \\ (3s^{2}+1)^{[4]} \ , \ (3s^{2}+2s+1)^{[2]} \ , \ 10s^{2}+s+1 \ , \ 3s^{2}+s+1 \ , \ 3s^{2}-s+1\bigg\}.
	\end{multline*}
\end{example}

\begin{Res}\label{ResRD1}
	Let $H=P_{4}$. Let $G_{i}$ be a $d_{i}-$regular graph of order $n_{i}$, $i=1,2$. Then the spectrum of $M_{G}(s)$,
	where $G=\bigvee\limits_{P_{4}}\{G_{1},G_{2},G_{2},G_{1}\}$, is given by
			\begin{equation*}
			\sigma(M_{G}(s))=\bigg(\sigma(M_{1}(s))^{[2]}-\{a(s)^{[2]}\}\bigg) \cup \bigg(\sigma(M_{2}(s))^{[2]}-\{b(s)^{[2]}\}\bigg) \cup \sigma(F_{4}(s)),
	\end{equation*}
\noindent
where $\sigma(M_{i}(s))^{[2]}$ means that each eigenvalue of $M_{i}(s)$, $i=1,2$, doubles its multiplicity and
\begin{align}\label{SpectF4}
\sigma(F_{4}(s))=\Bigg\{&\frac{1}{2}\bigg(a(s)+b(s)+sn_{2}\pm \sqrt{\big(a(s)-b(s)-sn_{2}\big)^{2}+4s^{2}n_{1}n_{2}}\bigg) \ , \notag \\ \notag \\ 
&\frac{1}{2}\bigg(a(s)+b(s)-sn_{2}\pm \sqrt{\big(a(s)-b(s)+sn_{2}\big)^{2}+4s^{2}n_{1}n_{2}}\bigg) \Bigg\},
\end{align}
being
\begin{align*}
 a(s)&=\lambda_{1}(M_{1}(s))=s^{2}(d_{1}+n_{2}-1)-s d_{1}+1, \\ \\ b(s)&=\lambda_{1}(M_{2}(s))=s^{2}(d_{2}+(n_{1}+n_{2})-1)-s d_{2}+1   
\end{align*}
and $\sigma(M_{i}(s))$, $i=1,2$, is given as in \eqref{SM}.
\end{Res}

\begin{proof}
Since $G=\bigvee\limits_{P_{4}}\{G_{1},G_{2},G_{2},G_{1}\}$, it follows that 
\begin{align*}
	F_{4}(s)=\begin{bmatrix}
		a(s) & -s\sqrt{n_{1}n_{2}} \\
		-s\sqrt{n_{1}n_{2}} & b(s) & -sn_{2} \\
		& -sn_{2} & b(s) & -s\sqrt{n_{1}n_{2}} \\
		& & -s\sqrt{n_{1}n_{2}} & a(s)
	\end{bmatrix}.
\end{align*}
From \eqref{det}, we have
\begin{align*}
\vert \lambda I-F_{4}(s)\vert &=(\lambda-a(s))f_{3}-s^{2}n_{1}n_{2}f_{2} \\ \\ &=(\lambda-a(s))\bigg((\lambda-b(s))f_{2}-s^{2}n^{2}_{2}(\lambda-a(s))\bigg)-s^{2}n_{1}n_{2}f_{2} \\ \\ &=f_{2}\bigg((\lambda-a(s))(\lambda-b(s))-s^{2}n_{1}n_{2}\bigg)-s^{2}n^{2}_{2}(\lambda-a(s))^{2} \\ \\ &=\bigg((\lambda-a(s))(\lambda-b(s))-s^{2}n_{1}n_{2}\bigg)^{2}-\bigg(sn_{2}(\lambda-a(s))\bigg)^{2},
\end{align*}
where $f_{3}$ and $f_{2}$ are the determinants of the leading principal submatrices of $\lambda I-F_{4}(s)$ of order $3$ and $2$, respectively.
Consequently, $\vert \lambda I-F_{4}(s)\vert=0$ if and only if
\begin{align*}
 \lambda^{2}-(a(s)+b(s)+sn_{2})\lambda+a(s)b(s)+sn_{2}a(s)-s^{2}n_{1}n_{2}&=0, \qquad \text{or} \\ \\
 \lambda^{2}-(a(s)+b(s)-sn_{2})\lambda+a(s)b(s)-sn_{2}a(s)-s^{2}n_{1}n_{2}&=0.
\end{align*}\ 

Thus, we find that the spectrum of $\sigma(F_{4}(s))$ is given by \eqref{SpectF4}. From Corollary \ref{CorCam}, the proof follows.
\end{proof}

\begin{example}\label{ExRD1}
Let $H=P_{4}$, $G_{1}=P_{2}$ and $G_{2}=C_{3}$. The graph $G=\bigvee\limits_{P_4}\{P_{2},C_{3},C_{3},P_{2}\}$ is depicted in Figure $5$:   
\begin{eqnarray*}
\definecolor{qqqqff}{rgb}{0.,0.,1.}
\begin{tikzpicture}
\draw [line width=1.2pt] (2.,2.)-- (3.,3.);
\draw [line width=1.2pt] (3.,3.)-- (3.,1.);
\draw [line width=1.2pt] (2.,2.)-- (3.,1.);
\draw [line width=1.2pt] (5.,3.)-- (6.,2.);
\draw [line width=1.2pt] (6.,2.)-- (5.,1.);
\draw [line width=1.2pt] (5.,3.)-- (5.,1.);
\draw [line width=1.2pt] (0.,1.)-- (0.,-1.);
\draw (8.,1.)-- (8.,-1.);
\draw (0.,1.)-- (2.,2.);
\draw (0.,1.)-- (3.,3.);
\draw (0.,1.)-- (3.,1.);
\draw (0.,-1.)-- (3.,1.);
\draw (3.,3.)-- (0.,-1.);
\draw (2.,2.)-- (0.,-1.);
\draw (3.,3.)-- (5.,3.);
\draw (3.,3.)-- (5.,1.);
\draw (3.,3.)-- (6.,2.);
\draw (3.,1.)-- (5.,3.);
\draw (3.,1.)-- (5.,1.);
\draw (3.,1.)-- (6.,2.);
\draw (2.,2.)-- (5.,3.);
\draw (2.,2.)-- (6.,2.);
\draw (2.,2.)-- (5.,1.);
\draw (6.,2.)-- (8.,1.);
\draw (6.,2.)-- (8.,-1.);
\draw (5.,3.)-- (8.,1.);
\draw (5.,3.)-- (8.,-1.);
\draw (5.,1.)-- (8.,1.);
\draw (5.,1.)-- (8.,-1.);
\begin{scriptsize}
\draw [fill=qqqqff] (2.,2.) circle (1.5pt);
\draw [fill=qqqqff] (3.,1.) circle (1.5pt);
\draw [fill=qqqqff] (3.,3.) circle (1.5pt);
\draw [fill=qqqqff] (6.,2.) circle (1.5pt);
\draw [fill=qqqqff] (5.,1.) circle (1.5pt);
\draw [fill=qqqqff] (5.,3.) circle (1.5pt);
\draw [fill=qqqqff] (8.,1.) circle (1.5pt);
\draw [fill=qqqqff] (8.,-1.) circle (1.5pt);
\draw [fill=qqqqff] (0.,1.) circle (1.5pt);
\draw [fill=qqqqff] (0.,-1.) circle (1.5pt);
\end{scriptsize}
\end{tikzpicture}&\\ \\
\text{{\bf Figure 5}. The graph $G=\bigvee\limits_{P_4}\{P_{2}, C_{3}, C_{3}, P_{2}\}$.}
\end{eqnarray*}
Taking into account \eqref{SM} and Result \ref{ResRD1}, we explicitly obtain the spectrum of $M_{G}(s)$ as follows:	
    \begin{align*}
		\sigma(M_{G}(s))=\bigg\{&\frac{1}{2}\left(9s^{2}+2\pm s\sqrt{9s^{2}+12s+28}\right) \ , \ \frac{1}{2}\left(9s^{2}-6s+2\pm s\sqrt{9s^{2}-24s+40}\right) \ , \\ \\ &(3s^{2}+s+1)^{[2]} \ , \ (6s^{2}-s+1)^{[2]} \ , \ (6s^{2}+s+1)^{[2]}\bigg\}.
	\end{align*}
\end{example}

\subsubsection{A particular case of Corollary \ref{CorCic}}
By proceeding with reasoning analogous to that used in Result \ref{ResRD1}, we get:
\begin{Res}\label{ResRD2}
	Let $H=C_{4}$. Let $G_{i}$ be a $d_{i}-$regular graph of order $n_{i}$, $i=1,2$. Then the spectrum of $M_{G}(s)$,
	where $G=\bigvee\limits_{C_{4}}\{G_{1},G_{2},G_{2},G_{1}\}$, is given by
	then
		\begin{equation*}
			\sigma(M_{G}(s))=\bigg(\sigma(M_{1}(s))^{[2]}-\{a(s)^{[2]}\}\bigg) \cup \bigg(\sigma(M_{2}(s))^{[2]}-\{b(s)^{[2]}\}\bigg) \cup \sigma(F_{4}(s)),
	\end{equation*}
\noindent
where $\sigma(M_{i}(s))^{[2]}$ means that each eigenvalue of $M_{i}(s)$, $i=1,2$, doubles its multiplicity and
\begin{align}\label{SpectF4}
\sigma(F_{4}(s))=\Bigg\{&\frac{1}{2}\bigg(a(s)+b(s)+s(n_{1}+n_{2})\pm \sqrt{\big(a(s)-b(s)+s(n_{1}-n_{2})\big)^{2}+4s^{2}n_{1}n_{2}}\bigg) \ , \notag \\ \notag \\ 
&\frac{1}{2}\bigg(a(s)+b(s)-s(n_{1}+n_{2})\pm \sqrt{\big(a(s)-b(s)-s(n_{1}-n_{2})\big)^{2}+4s^{2}n_{1}n_{2}}\bigg) \Bigg\},
\end{align}
being
\begin{align*}
 a(s)&=\lambda_{1}(M_{1}(s))=s^{2}(d_{1}+n_{2}-1)-s d_{1}+1, \\ \\ b(s)&=\lambda_{1}(M_{2}(s))=s^{2}(d_{2}+(n_{1}+n_{2})-1)-s d_{2}+1   
\end{align*}
and $\sigma(M_{i}(s))$, $i=1,2$, is given as in \eqref{SM}.
\end{Res}

\begin{example}\label{ExRD1}
Let $H=C_{4}$, $G_{1}=P_{2}$ and $G_{2}=C_{3}$. The graph $G=\bigvee\limits_{C_4}\{P_{2},C_{3},C_{3},P_{2}\}$ is depicted in Figure $6$:   
\begin{eqnarray*}
\definecolor{qqqqff}{rgb}{0.,0.,1.}
\begin{tikzpicture}
\draw [line width=1.2pt] (2.,2.)-- (3.,3.);
\draw [line width=1.2pt] (3.,3.)-- (3.,1.);
\draw [line width=1.2pt] (2.,2.)-- (3.,1.);
\draw [line width=1.2pt] (5.,3.)-- (6.,2.);
\draw [line width=1.2pt] (6.,2.)-- (5.,1.);
\draw [line width=1.2pt] (5.,3.)-- (5.,1.);
\draw [line width=1.2pt] (1.,-1.)-- (2.,-2.);
\draw [line width=1.2pt] (7.,-1.)-- (6.,-2.);
\draw (1.,-1.)-- (2.,2.);
\draw (1.,-1.)-- (3.,3.);
\draw (1.,-1.)-- (3.,1.);
\draw (2.,-2.)-- (3.,1.);
\draw (3.,3.)-- (2.,-2.);
\draw (2.,2.)-- (2.,-2.);
\draw (3.,3.)-- (5.,3.);
\draw (3.,3.)-- (5.,1.);
\draw (3.,3.)-- (6.,2.);
\draw (3.,1.)-- (5.,3.);
\draw (3.,1.)-- (5.,1.);
\draw (3.,1.)-- (6.,2.);
\draw (2.,2.)-- (5.,3.);
\draw (2.,2.)-- (6.,2.);
\draw (2.,2.)-- (5.,1.);
\draw (6.,2.)-- (7.,-1.);
\draw (6.,2.)-- (6.,-2.);
\draw (5.,3.)-- (7.,-1.);
\draw (5.,3.)-- (6.,-2.);
\draw (5.,1.)-- (7.,-1.);
\draw (5.,1.)-- (6.,-2.);
\draw (1.,-1.)-- (7.,-1.);
\draw (1.,-1.)-- (6.,-2.);
\draw (2.,-2.)-- (6.,-2.);
\draw (2.,-2.)-- (7.,-1.);
\begin{scriptsize}
\draw [fill=qqqqff] (2.,2.) circle (1.5pt);
\draw [fill=qqqqff] (3.,1.) circle (1.5pt);
\draw [fill=qqqqff] (3.,3.) circle (1.5pt);
\draw [fill=qqqqff] (6.,2.) circle (1.5pt);
\draw [fill=qqqqff] (5.,1.) circle (1.5pt);
\draw [fill=qqqqff] (5.,3.) circle (1.5pt);
\draw [fill=qqqqff] (7.,-1.) circle (1.5pt);
\draw [fill=qqqqff] (6.,-2.) circle (1.5pt);
\draw [fill=qqqqff] (1.,-1.) circle (1.5pt);
\draw [fill=qqqqff] (2.,-2.) circle (1.5pt);
\end{scriptsize}
\end{tikzpicture}&\\ \\
\text{{\bf Figure 6}. The graph $G=\bigvee\limits_{C_4}\{P_{2}, C_{3}, C_{3}, P_{2}\}$.}
\end{eqnarray*}
Taking into account \eqref{SM} and Result \ref{ResRD2}, we explicitly obtain the spectrum of $M_{G}(s)$ as follows:	
    \begin{align*}
		\sigma(M_{G}(s))=\bigg\{&\frac{1}{2}\left(9s^{2}-8s+2\pm s\sqrt{9s^{2}-12s+28}\right) \ , \ \frac{1}{2}\left(9s^{2}+2s+2\pm s\sqrt{27s^{2}+72}\right) \ , \\ \\ &(3s^{2}+s+1)^{[2]} \ , \ (6s^{2}-s+1)^{[2]} \ , \ (6s^{2}+s+1)^{[2]}\bigg\}.
	\end{align*}
\end{example}

\vspace{0.5cm}

\section*{Acknowledgements.}
The research of R. C. Díaz was supported by ANID Programa Fondecyt de Iniciaci\'on en Investigaci\'on 2024, N° 11240142, Chile. Elismar Oliveira and Vilmar Trevisan are partially supported by MATH-AMSUD under project GSA, financed by CAPES - process 88881.694479/2022-01, and by CNPq grant 408180/2023-4. Vilmar Trevisan also acknowledges the suport of CNPq grant 310827/2020-5.

\end{document}